\documentclass{mcom-l}
\newtheorem{theorem}{Theorem}[section]
\newtheorem{lemma}[theorem]{Lemma}
\newcommand{\W}{{\mathcal W}}
\theoremstyle{definition}

\newcommand{\p}{{\mathcal P}}
\theoremstyle{remark}

\usepackage{color}
\usepackage{graphicx}
\numberwithin{equation}{section}

\newcommand{\Uref}{U_{\text{ref}}}

\newcommand{\Bac}{{^C}\hspace{-0.03in}\partial_t^{\alpha}}
\newcommand{\I}{\mathcal{I}}

\begin{document}

\title[Error analysis for FEM  for fractional diffusion equations]{ FEM  for time-fractional diffusion equations, novel optimal error analyses}


\author{Kassem Mustapha}
\address{\thanks{Acknowledgments: The support of  the KFUPM  through the  project No. FT151002
 is gratefully acknowledged.} Department of Mathematics and Statistics, King Fahd University
 of Petroleum and Minerals, Dhahran, 31261, Saudi Arabia.}
\email{kassem@kfupm.edu.sa}

\subjclass[2010]{Primary }

\date{}

\dedicatory{}

\begin{abstract}
   A semidiscrete Galerkin finite element  method applied to time-fractional diffusion equations with time-space  dependent diffusivity on bounded convex spatial  domains will be studied. The main focus is on  achieving optimal error results with respect to both the convergence order of the approximate solution and the regularity of the initial data.   By using novel energy arguments,   for each fixed time  $t$, optimal error bounds  in the spatial  $L^2$- and $H^1$-norms are derived for both cases: smooth and nonsmooth initial data.  Some numerical results will be provided at the end.
\end{abstract}

\maketitle

\section{Introduction}
In this work, we consider the spatial discretisation via Galerkin finite elements of  the following   time-fractional diffusion problem:
find $u=u(x,t)$ so that
\begin{subequations}\label{a}
\begin{alignat}{2}\label{a1}
\Bac &u(x,t) -{\rm div}( \kappa_\alpha (x,t)\nabla u(x,t))
=0 &&\quad\mbox{ in }\Omega\times (0,T],
\\  \label{a2}
&u(x,t)= 0 &&\quad\mbox{ on }\partial\Omega\times (0,T],
\\   \label{a3}
&u(x,0)=u_0(x) &&\quad\mbox{ in }\Omega,
\end{alignat}
\end{subequations}
where  $\Omega$ is a bounded, convex polygonal domain in $\mathbb{R}^d$ ($d\ge 1$)   with
boundary $\partial \Omega$, $\kappa_\alpha$ and  $u_0$ are given functions defined on their
respective domains.  Here,  $\Bac$ is the Caputo  time-fractional
derivative  defined by: for $0<\alpha<1$,
\begin{equation} \label{Ba}
\Bac \varphi(t):=\I^{1-\alpha}\varphi'(t):=\int_0^t\omega_{1-\alpha}(t-s)\varphi'(s)\,ds,~~\text{with} ~~
\omega_{1-\alpha}(t):=\frac{t^{-\alpha}}{\Gamma(1-\alpha)},
\end{equation}
where $\varphi'$ denotes the (partial) time derivative of
$\varphi$ and for $\nu>0$,  $\I^\nu$ is the Riemann--Liouville time-fractional integral operator of order $\nu$ which reduces to the classical definite integral when $\nu$ is a positive integer. The  diffusivity coefficient $ \kappa_\alpha $ satisfies the positivity property:
\begin{equation}\label{eq: A positive}
0<\kappa_{\min} \le   \kappa_\alpha (x,t) \le  \kappa_{\max}<\infty\quad { \rm for
} ~~(x,t) \in \overline \Omega\times [0,T].\end{equation}
 Numerical solutions for time fractional diffusion problem \eqref{a} with  constant or space-dependent diffusion parameter  $ \kappa_\alpha $  have been studied by various authors over the last decade. For  finite difference (including alternating direction implicit schemes) and finite element (conforming and nonconforming)  schemes, we refer to  \cite{BrunnerYamamoto2010, ChenLiuLiuChenAnhTurnerBurrage2014, Cui2012, JLZ2013, JLPZ2015, LinXu2007, MurilloYuste2013, StynesORiordanGracia2017, ZengLiLiuTurner2013, ZhangLiao2014, ZhaoXu2014, Zhaoetal2016} and related references therein. Discontinuous Galerkin (DG) methods  (including local DG and hybridizable DG schemes) were investigated in  \cite{MustaphaNourCockburn2016, MustaphaAbdallahFurati2014, XuZheng2013}, and in  \cite{LiXu2009,ZhaoSun2011} the spectral method was studied. The convergence  analyses  in most of these studies required the solution $u$ of problem \eqref{a}  to be sufficiently regular
  including at $t = 0$ which is not practically the case.

   Having time dependent  variable diffusivity $ \kappa_\alpha $ in  the fractional diffusion problem \eqref{a} is indeed very interesting and also practically important. The numerical solutions  of \eqref{a} were  considered by a few authors only.  For {\em  one-dimensional} spatial  domain $\Omega$, a finite difference scheme was proposed and analyzed by Alikhanov \cite{Alikhanov2015}.  In the error analysis, the continuous solution  $u$ was assumed to be smooth including at $t=0$.  In  \cite{MustaphaAbdallahFuratiNour2016},  a piecewise linear time-stepping DG method  combined with the standard Galerkin finite element  scheme   in space was investigated. The convergence of the scheme had been proven assuming that  $u$ is sufficiently regular. Consequently, the convergence results in these papers are not valid if the initial data $u_0$ is not sufficiently regular where some compatibility conditions are aslo required.

   For constant diffusivity $ \kappa_\alpha $, Jin et al.  \cite{JLZ2013}  studied the error analysis of the spatial semidiscrete piecewise linear Galerkin  finite element  scheme for problem \eqref{a}.  Over a quasi-uniform spatial mesh,  quasi-optimal convergence order results (but optimal with respect to the regularity of  the initial data $u_0$) were proved. The used error analysis (based on semigroup) approach can be extended for the case of space dependent parameter $ \kappa_\alpha $, however is not feasible when $\kappa_\alpha$ is a time or  a time-space dependent function.   Therefore, the optimality of the finite element error estimates with respect to the convergence order and to the solution smoothness expressed through the problem data $u_0$ is indeed missing,  even for constant $ \kappa_\alpha $. So, obtaining optimal finite element  error bounds for the case of  time-space dependent diffusivity $ \kappa_\alpha $ is definitely challenging.

      The aim of this work is to  show  optimal error estimates with respect to both the convergence order and
the regularity of the initial data $u_0$  of the semidiscrete  Galerkin method  for problem   \eqref{a} allowing both smooth and  nonsmooth $u_0$. For each $t \in (0,T]$, by using a  novel energy arguments approach, we show optimal convergence results in the spatial $L^2(\Omega)$- and $H^1(\Omega)$-norms over a (conforming) regular triangulation  mesh (need not be quasi-uniform). It is straight forward to extend our error analysis approach to allow for an inhomogenous source term or homogenous  Neumann boundary conditions in problem \eqref{a}.

Note, for time independent diffusivity $ \kappa_\alpha $, problem \eqref{a} can be rewritten as:
\begin{equation}
\label{eq: reimann} u'(x,t) - {^{R}{\rm D}}^{1-\alpha} {\rm div}( \kappa_\alpha (x)\nabla u(x,t))   = 0 \quad\mbox{ for } (x,t)\in \Omega\times (0,T],
\end{equation}
where  $^{R}{\rm D}^{1-\alpha} u:=\frac{\partial }{\partial t}( \I^\alpha u)$ is the Riemann--Liouville fractional derivative. Recently, Karaa et al. \cite{KaraaMustaphaPani2017} investigated  the error analysis of the Galerkin finite element  scheme applied to problem \eqref{eq: reimann}.
Using a delicate  energy argument,  optimal error bounds in $L^2(\Omega)$- and $H^1(\Omega)$-norms, and quasi-optimal in  $L^{\infty}(\Omega)$-norm
were  derived  for cases  of  smooth and nonsmooth initial data. Unfortunately, extending the considered approach for the case of time dependent diffusivity is not feasible.

 Outline of the paper. In Section \ref{sec:WRT}, the required regularity assumptions on the solution $u$ of problem \eqref{a} will be given. We also state and derive some technical results that will be used in our error analysis. In Section \ref{sec:Semi-discrete FE}, we introduce our semidiscrete Galerkin scheme for problem \eqref{a}  and recall some  error projection results from the existing literature.  In Section \ref{sec: LinftyL2}, optimal error estimates (with respect to both the convergence order and the regularity of $u_0$)  in the $L^2(\Omega)$-norm will be proved using  novel energy arguments, see Theorem \ref{thm: smooth and nonsmoothv2}. For $t\in (0,T]$ and when $u_0 \in \dot H^\delta(\Omega)$ (this Sobolev space will be defined in the next section),  an  $O(t^{-\alpha(2-\delta)/2}h^2)$ error estimate   is proved for $0\le \delta\le 2$ (that is, allowing both smooth and nonsmooth initial data),  $h$
denoting the maximum diameter of the spatial mesh elements. Furthermore, in the $H^1(\Omega)$-norm,  we show an optimal error bounded by $C h\, t^{\alpha(\delta-2)/2}\|u_0\|_\delta$ for  $0\le \delta \le 1$, and by
$C h\,\max\{h\, t^{-\alpha/2},1\} t^{\alpha(\delta-2)/2}\|u_0\|_\delta$ for  $1< \delta \le 2$, where  $C$ is a generic constant that may depend on $\alpha$, $T$,  and the
norms of $\kappa_\alpha$, $\kappa_\alpha'$ and $\kappa_\alpha''$, but is independent of the spatial mesh size element $h$.    The derived  optimal bounds in both $L^2(\Omega)$- and $H^1(\Omega)$-norms  provide  remarkable improvements of results obtained  by Jin {\it et al.} in \cite[Theorem 3.7]{JLZ2013}. Therein,  for a quasi-uniform mesh and assuming that the  parameter $ \kappa_\alpha $ is constant, an  $O(t^{-\alpha(2-\delta)/2}h^{2-m}|\log h|)$ error bound was derived in the $H^m(\Omega)$-norm ($m=0,1$) when $u_0 \in \dot H^\delta(\Omega)$  with $\delta=0,\,1,\,2.$  Finally, for the numerical illustration of these achievements for the case of constant and space-dependent diffusivity coefficient $\kappa_\alpha$,  various tests  were carried out in \cite{JLZ2013}. For the case of space-time dependent $\kappa_\alpha$,  one numerical example will be provided in Section~\ref{sec: numerical}.

\section {Regularity and technical results}\label{sec:WRT}

It is known that the solution $u$ of problem \eqref{a}    has singularity near $t=0$, even for smooth given data. In our error analysis, we assume that  for  $0\le p\le q \le 2$,
\begin{equation}\label{eq: regularity property}
\begin{aligned}
\|u(t)\|_q+t\|u'(t)\|_q&&\le  Ct^{\alpha(p-q)/2}\|u_0\|_p,
\end{aligned}
\end{equation}
 where    $\|\cdot\|_r$ denotes the norm on the Hilbert space $\dot H^r(\Omega)\subset L^2(\Omega)$ defined by
 \[\|v\|_r^2 =\|{{\mathcal L}}^{r/2}v\|^2 =\sum_{j=1}^\infty \lambda_j^r (v,\phi_j)^2,\quad {\rm with}~~~{{\mathcal L}}v:=-{\rm div}( \kappa_\alpha \nabla v),\]
 where $\{\lambda_j\}_{j=1}^\infty$ (with  $0<\lambda_1\le  \lambda_2\le  \ldots$) are the eigenvalues of the operator ${{\mathcal L}}$ (subject homogeneous Dirichlet boundary conditions)
and $\{\phi_j\}_{j=1}^\infty$ are the associated orthonormal eigenfunctions.
In the above definition, $(\cdot,\cdot)$  denotes the $L^2(\Omega)$-norm and $\|\cdot\|:=\|\cdot\|_0$ is the associated norm.  Note, $\dot H^r(\Omega)=H^r(\Omega)$ for $0\le r<1/2,$ however, for  a convex polygonal domain $\Omega$, $\dot H^r(\Omega)=\{w \in H^r(\Omega): w=0~{\rm on}~\partial \Omega\}$ when   $1/2<r \le 2,$ where $H^r(\Omega)$ (with $H^0(\Omega)=L^2(\Omega)$) is the standard Sobolev space.

Indeed, for time independent function $ \kappa_\alpha $, the above regularity assumption holds assuming that the domain $\Omega$ is convex, see  Theorems 4.1 and 4.2 in \cite{Mclean2010}. We conjecture that the same is true for a sufficiently regular time dependent  $ \kappa_\alpha $.

Next, we state some  properties of the fractional integral  operators  $\I^{\alpha}$, and derive some technical results that will be used later. By \cite[Lemma 3.1(ii)]{MustaphaSchoetzau2014},   it follows that for piecewise time continuous functions  $\varphi:[0,T] \to L^2(\Omega),$
\begin{equation}\label{eq: positive of Ia}
\int_0^T(\I^\alpha\varphi,\varphi)\,dt\ge \cos(\alpha \pi/2)\int_0^T\|\I^{\alpha/2}\varphi\|^2\,dt \ge  0~~{\rm  for}~~0<\alpha<1\,.
\end{equation}

Furthermore, by \cite[Lemma 3.1(iii)]{MustaphaSchoetzau2014} and the inequality $\cos(\alpha \pi/2)\ge 1-\alpha$, we obtain the following continuity property of $\I^\alpha$: for  $\varphi\,,\psi \in L^2((0,T);L^2(\Omega)),$
\begin{equation}\label{eq: continuity}
\int_0^t (\I^{ 1-\alpha}\varphi,\psi)ds\le \epsilon  \int_0^t (\I^{1-\alpha}\varphi,\varphi)ds+\frac{1}{4\epsilon\,\alpha^2} \int_0^t (\I^{1-\alpha}\psi,\psi)ds,\quad {\rm for}~~\epsilon >0.
\end{equation}

In our convergence analysis,  we also make use of the inequality below, where the proof follows from   \cite[Lemma 2.1]{LeMcLeanMustapha2016} and the positivity property of $\I^{1-\alpha}$ (by \eqref{eq: positive of Ia}). If  $\varphi':[0,T] \to L^2(\Omega)$ is a piecewise time continuous function, then we have
\begin{equation}\label{lmm}
\|\varphi(t)\|^2\le C t^\alpha \int_0^t (\I^{1-\alpha}\varphi',\varphi')\,ds,\quad {\rm for}~~~t>0\,.
\end{equation}

 Based on the generalized  Leibniz formula and the relation between Riemann--Liouville and Caputo  fractional derivatives, we show the identity in the next lemma.  For convenience, we use the  notations:
\[v_i(t):=t^i v(t),\quad  {\rm for}~~i=1,\,2.\]
\begin{lemma}\label{Ia}
Let   $0<\alpha<1$. The following holds: for $0\le t \le T$,
\[t^2\I^\alpha v'(t) = \I^\alpha v_2'(t)+2(\alpha-1) \I^\alpha v_1(t)   +\alpha(\alpha-1)\I^{1+\alpha}v(t)-t^2\omega_\alpha (t)v(0)\,.\]
\end{lemma}
\begin{proof}
Since $\I^\alpha v'(t) =(\I^\alpha v(t))'-\omega_\alpha(t)v(0)$,  the use of the   fractional Leibniz formula $t(\I^\alpha v(t))'=(\I^\alpha v_1(t))'+(\alpha-1)\I^\alpha v(t)$ (see  \cite{Osler1970}) and the equality $(\I^\alpha v_1(t))'=\I^\alpha v_1'(t)$ yield the following identity:
\begin{equation}\label{A1}
 t\I^\alpha v'(t) = \I^\alpha v_1'(t)+(\alpha-1)\I^\alpha v(t)-t\omega_\alpha(t)v(0).
 \end{equation}
Now,  multiplying both side of the above identity  by $t$ and applying the identity: $t\I^\alpha \phi(t) = \I^\alpha \phi_1(t)+\alpha\I^{1+\alpha}\phi(t)$ (see \cite[Lemma 4.1 (b)]{KaraaMustaphaPani2017} for the proof) twice,
\begin{multline*}
t^2\I^\alpha v'(t) = t\I^\alpha v_1'(t)+(\alpha-1) t\I^\alpha v(t)-t^2\omega_\alpha(t)v(0)\\
= [\I^\alpha (v_1')_1(t)+\alpha\I^{1+\alpha} v_1'(t)]+(\alpha-1)[\I^\alpha v_1(t)+\alpha\I^{1+\alpha} v(t)]-t^2\omega_\alpha(t)v(0)\,.
\end{multline*}
Since  $ (v_1')_1(t)=t v_1'(t)= v_2'(t)-v_1(t)$ and $\I^{1+\alpha} v_1'(t)=\I^\alpha v_1(t)$, the desired identity follows after simple simplifications.
\end{proof}

\begin{lemma}\label{lemma: tech1}
Let $g\ge 0$ be a   nondecreasing function of $t$.   \\
(i) If
\begin{equation}\label{sup-3-2-1}
 \int_0^t ( \I^{1-\alpha} v,v)\,ds +2\int_0^t (\I( \kappa_\alpha  w),w)\,ds \le g(t),\quad{\rm for}~~t>0,\end{equation}
for suitable functions $v$ and $w$, then for $  \kappa_\alpha ' \in L^\infty((0,T);L^\infty(\Omega))$, we have
\[
\int_0^t ( \I^{1-\alpha} v,v)\,ds +\|\I  w(t)\|^2 \le
Cg(t).\]
(ii) If
\begin{equation}\label{sup-3-2-1-1}
\int_0^t ( \I^{2-\alpha} v,\I v)\,ds +2\int_0^t (\I^2( \kappa_\alpha   w),\I w)\,ds \le g(t)\quad{\rm for}~~t>0,\end{equation}
for suitable functions $v$ and $w$, then for $ \kappa_\alpha',  \kappa_\alpha'' \in L^\infty((0,T);L^\infty(\Omega))$, we have
\[
\int_0^t ( \I^{2-\alpha} v,\I v)\,ds +\|\I^2 w(t)\|^2 \le
Cg(t).\]
\end{lemma}
\begin{proof}
Let $w_I(t):=\I w(t)=\int_0^t w(s)\,ds$. Since $\I( \kappa_\alpha w)= \kappa_\alpha \, w_I-\I( \kappa_\alpha 'w_I)$, an integration by parts yields
\[
\begin{aligned}
2\int_0^t &(\I( \kappa_\alpha w),w)\,ds = \int_0^t \big( \kappa_\alpha ,(w_I^2)'\big)\,ds-2\int_0^t \big(\I( \kappa_\alpha 'w_I),w_I'\big)\,ds\\
& = \big( \kappa_\alpha (t),w_I^2(t)\big) - \int_0^t \big( \kappa_\alpha ',w_I^2\big)\,ds-2\big(\I( \kappa_\alpha 'w_I)(t),w_I(t)\big) +2\int_0^t \big( \kappa_\alpha ',w_I^2\big)\,ds\\
& = \big( \kappa_\alpha (t),w_I^2(t)\big) -2\big(\I( \kappa_\alpha 'w_I)(t),w_I(t)\big)+ \int_0^t \big( \kappa_\alpha ',w_I^2\big)\,ds\,.
\end{aligned}
\]
Therefore, by inserting this in \eqref{sup-3-2-1}, then  using the positivity assumption on the diffusion coefficient $ \kappa_\alpha $, \eqref{eq: A positive}, and the Cauchy-Schwarz inequality, we conclude that
\[
\begin{aligned}
\int_0^t ( \I^{1-\alpha}v,v)\,ds +\|w_I(t)\|^2 &\le
C g(t)+C\int_0^t \|w_I\|ds\|w_I(t)\|+ C\int_0^t \|w_I\|^2\,ds\\
 &\le
C g(t)+\frac{1}{2}\|w_I(t)\|^2+ C\int_0^t \|w_I\|^2\,ds.\end{aligned}\]
 Thus,
\[\int_0^t ( \I^{1-\alpha}v,v)\,ds +\|w_I(t)\|^2 \le
Cg(t)+C\int_0^t \|w_I\|^2ds.\]
  Since $\int_0^t ( \I^{1-\alpha}v,v)\,ds\ge 0$ by the positivity property in \eqref{eq: positive of Ia},  an application of the continuous version of Gronwalls inequality yields the first desired result.

To show $(ii)$, we let $w_{II}:=\I^2  v$. Since $\I( \kappa_\alpha  w_{II}'')= \kappa_\alpha  w_{II}' -\I( \kappa_\alpha ' w_{II}')$,
\begin{align*}
\I^2( \kappa_\alpha w_{II}'')(s)&=\I( \kappa_\alpha w_{II}')(s)-\I^2( \kappa_\alpha ' w_{II}')(s)\\
&=  \kappa_\alpha (s)w_{II}(s)-\I( \kappa_\alpha 'w_{II})(s)-\I^2( \kappa_\alpha 'w_{II}')(s)\\
&=  \kappa_\alpha (s)w_{II}(s)-2\I( \kappa_\alpha 'w_{II})(s)+\I^2( \kappa_\alpha ''w_{II})(s)\,.\end{align*}
Thus, an integration by parts yields
\[
\begin{aligned}
2&\int_0^t (\I^2( \kappa_\alpha  w),\I w)\,ds
=2\int_0^t (\I^2( \kappa_\alpha w_{II}''),w_{II}')\,ds \\
&= \int_0^t \Big( \kappa_\alpha ,\big(w_{II}^2\big)'\Big)\,ds-2\int_0^t \big(2\I( \kappa_\alpha 'w_{II})-\I^2( \kappa_\alpha ''w_{II}),w_{II}'\big)\,ds\\
& = \big( \kappa_\alpha (t),w_{II}^2(t)\big) - \int_0^t \big( \kappa_\alpha ',w_{II}^2\big)\,ds\\
&\quad -2\big(2\I( \kappa_\alpha 'w_{II})(t)-\I^2( \kappa_\alpha ''w_{II})(t),w_{II}(t)\big)
 +2\int_0^t \big(2 \kappa_\alpha 'w_{II}-\I( \kappa_\alpha ''w_{II}),w_{II}\big)\,ds\\
& = \big( \kappa_\alpha (t),w_{II}^2(t)\big) + 3\int_0^t \big( \kappa_\alpha ',w_{II}^2\big)\,ds\\
&\quad -2\big(2\I( \kappa_\alpha 'w_{II})(t)-\I^2( \kappa_\alpha ''w_{II})(t),w_{II}(t)\big)
 -2\int_0^t \big(\I( \kappa_\alpha ''w_{II}),w_{II}\big)\,ds\,.
\end{aligned}\]
 Now, by proceeding as in the proof of $(i)$, we obtain the second desired result.
\end{proof}

\section{Finite element discretization} \label{sec:Semi-discrete FE}
This section focuses on the  spatial semidiscrete Galerkin finite element   scheme for the time fractional diffusion   problem \eqref{a}. Let  $\mathcal{T}_h$ be a
family of  shape-regular  triangulations (made of simplexes $K$) of
the  domain $\overline{\Omega}$ and let $h=\max_{K\in \mathcal{T}_h}(\mbox{diam}K),$ where $h_{K}$ denotes
the diameter  of the element  $K.$ Let $S_h \in H^1_0(\Omega)$  denote the
usual space of continuous, piecewise-linear functions on $\mathcal{T}_h$ that  vanish on $\partial \Omega$.

The weak formulation for problem  \eqref{a} is
to find   $u:( 0,T]\longrightarrow H^1_0(\Omega)$ such that
\begin{equation} \label{weak}
(\Bac u,v )+ A(u,v )=  0\quad
\forall v\in H^1_0(\Omega)
\end{equation}
with given  $ u(0)=u_0.$ Here  $A(\cdot, \cdot)$ is the  bilinear form associated with the elliptic operator ${{\mathcal L}}$, i.e., $A(v,w)=( \kappa_\alpha \nabla v, \nabla w)$, which is symmetric positive definite on  the Sobolev space $H_0^1(\Omega)$ for each fixed $t\in [0,T]$.

Now, the semidiscrete  scheme  for \eqref{a} is
to seek  $u_h:(0,T]\longrightarrow S_h$ such that
\begin{equation} \label{semi}
(\Bac u_h,\chi)+ A(u_h,\chi)=  0\quad
\forall \chi\in S_h,
\end{equation}
with given  $u_h(0):=u_{h0}=P_h u_0,$ where $P_h :L^2(\Omega)\rightarrow S_h$ denotes the $L^2$-projection defined by $(P_h v-v, \chi)= 0$ for all $\chi\in S_h.$   Indeed, for initial data $u_0\in {\dot H}^1(\Omega)$, one may choose instead  $u_h(0)=R_hu_0$, where   $R_h : H_0^1(\Omega) \rightarrow S_h $ is the Ritz projection defined by the following relation:
$A(R_h v-v, \chi)= 0$ for all $\chi\in S_h.$

For the  error analysis, we use the following decomposition:
\begin{equation} \label{eq: decomposition} u-u_h=(u- R_h u)-( u_h-R_hu)=:\rho-\theta.\end{equation}
For $t\in (0,T]$, from the projection error  estimates
\cite[(3.2) and (3.3)]{LuskinRannacher1982} and the regularity assumption in \eqref{eq: regularity property}, for $0\le \delta\le m$ with $m=1,2,$
\begin{equation}\label{rho-estimate}
\|\rho(t)\|+h\|\rho(t)\|_1\leq C h^m \|u(t)\|_m\le Ch^m t^{\alpha(\delta-m)/2}\|u_0\|_\delta,\end{equation}
 and
 \begin{equation}\label{rho-estimate1}
 \|\rho'(t)\|+h\|\rho'(t)\|_1\leq C  h^m\Bigl(\|u(t)\|_m+\|u'(t)\|_m\Bigr)\le Ch^m t^{\alpha(\delta-m)/2-1}\|u_0\|_\delta.\end{equation}

   We need to assume that $\kappa_\alpha \in L^\infty((0,T);W^{1,\infty}(\Omega))$ in \eqref{rho-estimate},  and in addition to this, $\kappa_\alpha' \in L^\infty((0,T);W^{1,\infty}(\Omega))$ in \eqref{rho-estimate1}. Noting that, when $m=1$, for the $H^1$-norm projection error  estimates,   $W^{1,\infty}(\Omega)$ can be replaced with $L^\infty(\Omega)$ in these assumptions.

Therefore, for later use, we have
\begin{multline}\label{eq: bound of B1}
\|\I^{1-\alpha}\rho(t)\|+\|\I^{1-\alpha}\rho_1'(t)\|\leq C\int_0^t (t-s)^{-\alpha}\big[\|\rho(s)\|+s\|\rho'(s)\|\big]\,ds\\
\leq Ch^{m}  \int_0^t (t-s)^{-\alpha}  s^{\alpha(\delta-m)/2}\,ds\,\|u_0\|_\delta\\
= Ch^{m}  t^{1-\alpha+\alpha(\delta-m)/2}\|u_0\|_\delta,~~ {\rm for}~ 0\le \delta\le m,~~{\rm with}~~m=1,2\,.
\end{multline}
In a similar fashion,  for $t\in (0,T]$, we have
\begin{multline}\label{eq: bound of B2}
\|\I^{1-\alpha}\rho_2'(t)\|+  \|\I^{1-\alpha}\rho_1(t)\|
\leq C\int_0^t (t-s)^{-\alpha}\big[s\|\rho(s)\|+s^2\|\rho'(s)\|\big]\,ds\\
\leq C\,h^2 \int_0^t  (t-s)^{-\alpha}  s^{1+\alpha(\delta-2)/2} \|u_0\|_\delta\\
\le  C\,h^2  t^{2-\alpha+\alpha(\delta-2)/2} \|u_0\|_\delta,\quad {\rm for}~~ 0\le \delta\le2\,.
\end{multline}

Via an energy argument approach, we estimate $\theta$ (and consequently the finite element  error) in the next section.

\section {Error estimates} \label{sec: LinftyL2}
 This section is devoted  to derive optimal error bounds from the Galerkin  approximation in both $L^2(\Omega)$- and $H^1(\Omega)$-norms, for the case of
 smooth and nonsmooth initial data  $u_0$. More precisely, for $t\in (0,T]$ and for $u_0\in \dot H^\delta(\Omega)$, we show
$$ \|(u-u_h)(t)\| + h\|\nabla(u-u_h)(t)\|   \leq
 C h^2 t^{\alpha(\delta-2)/2}\|u_0\|_\delta\quad{\rm for}~~0\le \delta \le 2,$$
see Theorem \ref{thm: smooth and nonsmoothv2}, and the estimates in \eqref{eq: H1 bound 0 delta 1} and \eqref{eq: H1 bound 1 delta 2}. Noting that, for the $H^1(\Omega)$-norm error,  the spatial mesh is assumed to be quasi-uniform when  $1< \delta \le 2$.

The  estimate of $\theta_1$ in the  next lemma   plays a crucial role in achieving our error bounds.
\begin{lemma} \label{eq: bound of I 1-alpha Theta1}
Assume that $\kappa_\alpha'\,,\kappa_\alpha'' \in L^\infty((0,T);L^\infty(\Omega)).$  For $0\le t\le T$,  we have
\[\int_0^t (\I^{1-\alpha}\theta_1,\theta_1) +\|\I (\nabla \theta_1)(t)\|^2 \le
C  \int_0^t [|(\I^{1-\alpha}\rho_1,\rho_1)|+|(\I^{2-\alpha}\rho,\I\rho)|]ds\,.\]
\end{lemma}
\begin{proof} From (\ref{weak}) and (\ref{semi}), the error decomposition $u-u_h=\rho-\theta$ in \eqref{eq: decomposition}, and the property of the Ritz projection, we obtain
\begin{equation} \label{sup-2}
(\I^{1-\alpha}\theta',\chi)+A(\theta,\chi)=
(\I^{1-\alpha}\rho',\chi)\quad \forall~\chi \in S_h.
\end{equation}
 Multiplying both sides of \eqref{sup-2} by $t$, gives
\[
(t\I^{1-\alpha}\theta',\chi)+A(\theta_1,\chi)=
(t\I^{1-\alpha}\rho',\chi)\,.
\]
Hence,  by the identity in \eqref{A1} and   the equality   $(u_0-u_{h0},\chi)= 0$,  we obtain
\begin{equation}\label{eq: Theta1}
\begin{aligned}
(\I^{1-\alpha}\theta_1'-\alpha\I^{1-\alpha}\theta,\chi)+A(\theta_1,\chi)&=
(\I^{1-\alpha}\rho_1'-\alpha\I^{1-\alpha}\rho,\chi)\,.
\end{aligned}
\end{equation}
Integrating \eqref{eq: Theta1} in time   and rearranging the terms to get
\[
\begin{aligned}
(\I^{1-\alpha}\theta_1,\chi)+(\I(\kappa_\alpha \nabla \theta_1),\nabla \chi)&=
(\I^{1-\alpha}\rho_1-\alpha\I^{1-\alpha}(\I\rho)+\alpha\I^{1-\alpha}(\I\theta),\chi),
\end{aligned}
\]
for all $\chi \in S_h.$ Choosing $\chi=\theta_1(s)\in S_h$, and then  integrating again in time  and using the continuity property in \eqref{eq: continuity} (with $\epsilon=\frac{1}{4}$) for the three terms on the right-hand side, we observe that
\begin{equation}\label{eq: bound of theta1}
\begin{aligned}
\int_0^t [(\I^{1-\alpha}&\theta_1,\theta_1)+(\I(\kappa_\alpha \nabla \theta_1),\nabla \theta_1)]ds\\
&\quad \le
C  \int_0^t [(\I^{1-\alpha}\rho_1,\rho_1)+(\I^{1-\alpha}(\I\rho),\I\rho)+(\I^{1-\alpha}(\I\theta),\I\theta)]ds\,.
\end{aligned}
\end{equation}
To estimate the last term on the right-hand side of \eqref{eq: bound of theta1}, we first integrate both sides of \eqref{sup-2} in time and use the identity $\I^{2-\alpha} v'(t)=\I^{1-\alpha} v(t)-\omega_{2-\alpha}(t) v(0)$,
\[
(\I^{1-\alpha}\theta,\chi)+(\I( \kappa_\alpha \nabla \theta),\nabla \chi)=
(\I^{1-\alpha}\rho-\omega_{2-\alpha}(t)[u_0-u_{h0}],\chi)\quad \forall~\chi \in S_h.
\]
Since  $(u_0-u_{h0},\chi)=(u_0-P_h u_0,\chi)= 0$,
\begin{equation} \label{sup-3 ph}
(\I^{1-\alpha}\theta,\chi)+(\I( \kappa_\alpha \nabla \theta),\nabla \chi)=
(\I^{1-\alpha}\rho,\chi)\quad \forall~\chi \in S_h.
\end{equation}
Integrating both sides of \eqref{sup-3 ph} yields
\begin{equation} \label{sup-3 ph-1}
(\I^{1-\alpha}(\I\theta),\chi)+(\I^2(\kappa_\alpha\nabla \theta),\nabla \chi)=
(\I^{1-\alpha}(\I\rho),\chi)\quad \forall~\chi \in S_h.
\end{equation}
Setting  $\chi=\I\theta(s) \in S_h$ and then, integrating over the time variable  and applying the continuity property of $\I^{1-\alpha}$ (with $\epsilon=\frac{1}{2}$), we  find that
\begin{multline*}
\int_0^t ( \I^{1-\alpha}(\I\theta),\I\theta)\,ds +\int_0^t (\I^2( \kappa_\alpha \nabla \theta),\I(\nabla \theta))\,ds =
\int_0^t (\I^{1-\alpha}(\I\rho),\I\theta)ds\\
\le \frac{1}{2}\int_0^t ( \I^{1-\alpha}(\I\theta),\I\theta)\,ds+
C  \int_0^t (\I^{1-\alpha}(\I\rho),\I \rho)ds\,.
\end{multline*}
After simplification,   an application of Lemma \ref{lemma: tech1} $(ii)$   gives
\begin{equation} \label{eq: bound of I 2-alpha}
\int_0^t ( \I^{2-\alpha}\theta,\I\theta)\,ds \le
C  \int_0^t |(\I^{2-\alpha}\rho,\I\rho)|ds.\end{equation}
Inserting  this bound in \eqref{eq: bound of theta1} gives
\[\begin{aligned}
\int_0^t [(\I^{1-\alpha}\theta_1,\theta_1)+&(\I( \kappa_\alpha  \nabla \theta_1),\nabla \theta_1)]ds\le
C  \int_0^t \Big[|(\I^{1-\alpha}\rho_1,\rho_1)|+|(\I^{2-\alpha}\rho,\I\rho)|\Big]ds\,.
\end{aligned}\]
Finally,  an application of Lemma \ref{lemma: tech1} $(i)$ yields the desired bound. \end{proof}

 Now, we are ready to derive an estimate of $\theta$ that will be used later to derive optimal finite element error bounds for the case of smooth and  nonsmooth initial data.
\begin{lemma}
 \label{H1-2} Assume that $ \kappa_\alpha ',  \kappa_\alpha '' \in L^\infty((0,T);L^\infty(\Omega))$. For $0\le t\le T$, the following estimate holds
\begin{multline*}
\|\theta(t)\|^2 + t^\alpha \|\nabla \theta(t)\|^2 \\ \le C  t^{\alpha-4} \int_0^t\big(\|\I^{1-\alpha}\rho_2'\|+  \|\I^{1-\alpha}\rho_1\|+  \|\I^{2-\alpha}\rho\|\big)\big(\|\rho_2'\|+\|\rho_1\|+\|\I\rho\|\big)\,ds\,.\end{multline*}
\end{lemma}
\begin{proof} Multiplying both sides of \eqref{sup-2} by $t^2$ gives
\[
(t^2\I^{1-\alpha}\theta',\chi)+A(\theta_2,\chi)=
(t^2\I^{1-\alpha}\rho',\chi)\quad \forall~\chi \in S_h\,.
\]
Using the  identity in  Lemma \ref{Ia} and   the fact that  $(u_0-u_{h0},\chi)= 0$ yields
\begin{multline}\label{eq: Theta1-2}
(\I^{1-\alpha}[\theta_2'-2\alpha \theta_1   -\alpha(1-\alpha)\I\theta],\chi)+A(\theta_2,\chi)=(\I^{1-\alpha}\eta,\chi),\quad  \forall~\chi \in S_h\,.
\end{multline}
where \begin{equation}\label{eta}
\eta=\rho_2'-2\alpha \rho_1   -\alpha(1-\alpha)\I\rho\,.\end{equation}
Rearranging the terms,
\[
\begin{aligned}
(\I^{1-\alpha}\theta_2',\chi)+A(\theta_2,\chi)&=
(\I^{1-\alpha}\eta,\chi)+2\alpha (\I^{1-\alpha}\theta_1,\chi)   +\alpha(1-\alpha)(\I^{1-\alpha}(\I\theta),\chi)\,.
\end{aligned}
\]
Setting  $\chi=\theta_2'(s)\in S_h$,  integrating over the time interval $(0,t)$, and using  the continuity property of $\I^{1-\alpha}$ in \eqref{eq: continuity} (for an appropriate choice of $\epsilon$)  for each term on the right-hand side,  we reach
\begin{multline*}
\int_0^t [(\I^{1-\alpha}\theta_2',\theta_2')+A(\theta_2,\theta_2')]\,ds \\
\le
 \frac{1}{2}\int_0^t (\I^{1-\alpha}\theta_2',\theta_2')\,ds
 +C
\int_0^t [(\I^{1-\alpha}\eta,\eta)+(\I^{1-\alpha}\theta_1,\theta_1)+(\I^{2-\alpha}\theta,\I\theta)]ds\,.\,.
\end{multline*}
Integration by parts  follows by using the positivity assumption of $\kappa_\alpha$ in \eqref{eq: A positive}, gives
\[
\begin{aligned}
2\int_0^t A(\theta_2,\theta_2')\,ds&=\int_0^t \big( \kappa_\alpha ,  ((\nabla \theta_2)^2)'\big)ds\\
&=( \kappa_\alpha (t),  (\nabla \theta_2)^2(t))-\int_0^t \big( \kappa_\alpha ', (\nabla \theta_2)^2\big)ds
\\ &\ge  \kappa_{\min} \|\nabla \theta_2(t)\|^2-\int_0^t \big( \kappa_\alpha ', (\nabla \theta_2)^2\big)ds,
\end{aligned}
\]
and hence,
\begin{multline*}
\int_0^t (\I^{1-\alpha}\theta_2',\theta_2')\,ds +\|\nabla\theta_2(t)\|^2 \le C\int_0^t \|\nabla \theta_2\|^2ds+\frac{1}{2}\int_0^t (\I^{1-\alpha}\theta_2',\theta_2')\,ds\\
 +C
\int_0^t [(\I^{1-\alpha}\eta,\eta)+(\I^{1-\alpha}\theta_1,\theta_1)+(\I^{2-\alpha}\theta,\I\theta)]ds\,.
\end{multline*}

Simplifying, and then using  \eqref{eq: bound of I 2-alpha} and Lemma \ref{eq: bound of I 1-alpha Theta1}, we obtain
\[
\int_0^t (\I^{1-\alpha}\theta_2',\theta_2')\,ds + \|\nabla \theta_2(t)\|^2\le C
\int_0^t  \big(|(\I^{1-\alpha}\eta,\eta)|+|(\I^{1-\alpha}\rho_1,\rho_1)|+|(\I^{2-\alpha}\rho,\I\rho)|\big)\,ds +C\int_0^t \|\nabla \theta_2\|^2ds\,.\]
Therefore,  applications of the  inequality in \eqref{lmm} and  the continuous version of  Gronwalls inequality yield
\[t^{-\alpha}\|\theta_2(t)\|^2 + \|\nabla \theta_2(t)\|^2 \le C
\int_0^t  \big(|(\I^{1-\alpha}\eta,\eta)|+|(\I^{1-\alpha}\rho_1,\rho_1)|+|(\I^{2-\alpha}\rho,\I\rho)|\big)\,ds\,.\]
The desired result follows immediately after using the fact that   $\theta(t)=t^{-2}\theta_2(t)$, the definition of $\eta$ in \eqref{eta} and the Cauchy-Schwarz inequality.
\end{proof}

In the next theorem, we show that the $L^2(\Omega)$-norm error from the spatial discretization by the scheme \eqref{semi} is  bounded by $C h^2 t^{\alpha(\delta-2)/2}\|u_0\|_\delta$  for $0\le \delta\le 2$.
 \begin{theorem} \label{thm: smooth and nonsmoothv2}
 Let $u$ be the solution of the time fractional diffusion problem  $(\ref{a})$ and let $u_h$  be the finite element solution defined by  $(\ref{semi})$, with $u_{h0}=P_h u_0$. Assume that  $\kappa_\alpha\,,\kappa_\alpha' \in L^\infty((0,T);W^{1,\infty}(\Omega))$ and
 $\kappa_\alpha''\in L^\infty((0,T);L^\infty(\Omega))$. Then,  for $t \in (0,T]$,  we have
$$
 \|(u-u_h)(t)\|  \leq
 C h^2 t^{\alpha(\delta-2)/2}\|u_0\|_\delta\quad{\rm for}~~0\le \delta \le 2\,.
$$
 \end{theorem}
\begin{proof}
By using the estimate in \eqref{eq: bound of B2} and  the projection  error  bounds  in \eqref{rho-estimate}--\eqref{eq: bound of B1} (with  $m=2$) and \eqref{eq: bound of B2}, we find that for $t\in (0,T]$,
\begin{multline*}
\int_0^t\Big[\|\I^{1-\alpha}\rho_2'\|\,\|\rho_2'\|+  \|\I^{1-\alpha}\rho_1\|\,\|\rho_1\|+  \|\I^{2-\alpha}\rho\|\,\|\I\rho\|\Big]ds\\
\le C\,h^4 \int_0^t s^{2-\alpha+\alpha(\delta-2)/2} s^{1+\alpha(\delta-2)/2}\,ds\,\|u_0\|_\delta^2\\
\le  C\,h^4  t^{4-\alpha+\alpha(\delta-2)} \|u_0\|_\delta^2,\quad{\rm for}~~~0\le \delta\le 2\,.
\end{multline*}
 Inserting this bound in the achieved estimate in Lemma \ref{H1-2} yields
  \begin{equation}\label{eq: last bound of theta}
\|\theta(t)\| + t^{\alpha/2} \|\nabla \theta(t)\| \le C\,h^2  t^{\alpha(\delta-2)/2} \|u_0\|_\delta,\quad{\rm for}~~~0\le \delta\le 2\,.\end{equation}
On the other hand, from the error projection estimate of $\rho$  in \eqref{rho-estimate} for $m=2$,
  \begin{equation}\label{eq: last bound of rho}
\|\rho(t)\|+h\|\nabla \rho(t)\|\le Ch^2 t^{\alpha(\delta-2)/2}\|u_0\|_\delta,\quad {\rm for}~~0\le \delta\le 2.\end{equation}
Therefore, the desired  error bounds follow from the decomposition $u-u_h=\rho-\theta$, and the above estimates.
\end{proof}

The $H^1(\Omega)$-norm convergence  will be discussed next.    By using  \eqref{rho-estimate}, \eqref{rho-estimate1} and \eqref{eq: bound of B1} but with  $m=1$,
\begin{equation*}
\begin{aligned}
\int_0^t\Big[\|\I^{1-\alpha}\rho_2'\|\,\|\rho_2'\|+  \|\I^{1-\alpha}\rho_1\|\,\|\rho_1\|+&  \|\I^{2-\alpha}\rho\|\,\|\I\rho\|\Big]ds\\
&\leq C\,h^2 \|u_0\|_\delta^2 \int_0^t s^{3-\alpha+\alpha(\delta-1)}\,ds\\
&\le  C\,h^2 t^{4-2\alpha+\alpha\delta}\|u_0\|_\delta^2,\quad {\rm for}~~ 0\le \delta\le  1\,.
\end{aligned}
\end{equation*}
Hence, by Lemma \ref{H1-2},
\[\|\nabla\theta(t)\|^2 \le C  \,h^2 t^{\alpha(\delta-2)}\|u_0\|_\delta^2,\quad {\rm for}~~ 0\le \delta\le 1.\]
Therefore, from the decomposition $u-u_h=\rho-\theta$, the above estimate, and \eqref{rho-estimate} with $m=1$, we reach the following optimal $H^1(\Omega)$-norm error bound:
\begin{equation}\label{eq: H1 bound 0 delta 1}
 \|\nabla(u-u_h)(t)\|
 \le C \,h\, t^{\alpha(\delta-2)/2}\|u_0\|_\delta,\quad {\rm for}~~ 0\le \delta\le 1\,.
\end{equation}
However, for $u_0 \in \dot H^\delta(\Omega)$ with $1< \delta\le 2,$ once again, from the decomposition $u-u_h=\rho-\theta$ and the estimates in \eqref{eq: last bound of theta} and \eqref{eq: last bound of rho},  we find that
$$
 \|\nabla(u-u_h)(t)\|
 \le C \,h\, t^{\alpha(\delta-2)/2}
\max\{h\, t^{-\alpha/2},1\}\|u_0\|_\delta,\quad {\rm for}~~ 1< \delta\le 2\,.
$$
 This  error bound is  optimal provided that   $h^2 \le t^\alpha$. Indeed, by assuming that the spatial  mesh is quasi-uniform, this optimality can also be preserved even if  $h^2 > t^\alpha$. To see this, we apply the inverse inequality and use the achieved estimate in  \eqref{eq: last bound of theta},
 \begin{equation}\label{eq: bound of theta adv-1}
\|\nabla \theta(t)\|\le  Ch^{-1}\|\theta(t)\|\le  C\,h^2 \, t^{\alpha(\delta-2)/2} \|u_0\|_\delta,\quad  {\rm for}~ 1<\delta\le 2\,.
\end{equation}
Hence, for $t\in (0,T]$, we have
\begin{equation}\label{eq: H1 bound 1 delta 2}
 \|\nabla(u-u_h)(t)\|
 \le C \,h\, t^{\alpha(\delta-2)/2}\|u_0\|_\delta,\quad {\rm for}~~ 1< \delta\le 2\,.
\end{equation}


\section{Numerical results}\label{sec: numerical}
The aim of this section is to validate the achieved theoretical results numerically for the case of time-space dependent variable coefficient $\kappa_\alpha$ and nonsmooth initial data $u_0.$  For smooth $u_0$, some numerical results were carried out in \cite{MustaphaAbdallahFuratiNour2016}. Furthermore, for time independent $\kappa_\alpha$, extensive numerical tests were carried out in \cite{JLZ2013}, where
the empirical convergence rates in all numerical experiments confirm the theoretical findings for both smooth and nonsmooth initial data.

To compute the finite element solution, time discretization via  a piecewise linear discontinuous Galerkin method will be considered \cite{MustaphaAbdallahFuratiNour2016}. For time levels
$0=t_0<t_1<t_2<\cdots<t_N=T$, we denote the $n$th step size by~$\tau_n=t_n-t_{n-1}$ and the associated subinterval by~$I_n=(t_{n-1},t_n)$, for $1\le n\le N$.  The maximum time step size is denoted by $\tau$. Let
\[
 \W=\{w \in L^2((0,T), S_h):~~w|_{I_{n}}\in
\p_1(S_h)~{\rm for}~1\le n\le N\}, \]  where $\p_1(S_h)$ denotes the space
of linear polynomials in the time variable $t$,  with coefficients in~$S_h$. Using the elementary identity $\Bac \varphi(t)=
^{R}{\rm D}^{\alpha} \varphi(t) - \omega_{1-\alpha}(t)\varphi(0)$ ($^{R}{\rm D}^{1-\alpha}$ is the Riemann--Liouville fractional derivative), the finite element scheme \eqref{semi} can be rewritten as: for $0<t\le T,$
\[
(^{R}{\rm D}^{\alpha} u_h,\chi)+ A(u_h,\chi)=  \omega_{1-\alpha}(t)(u_{h0},\chi)\quad
\forall \chi\in S_h\,.
\]
We approximate~$u_h(t_n)$ by $U^n:=U(t_n^-)$ where $U \in \p_1(S_h)$
satisfying
\[
\int_{I_n}\bigl[({\rm ^R D}^\alpha  U,X)
    +A\bigl( U, X\bigr)\bigr]\,dt
    =\int_{I_n}(\omega_{1-\alpha}(t)u_{h0},X)\,dt\quad
\forall X\in \p_1(S_h),~{\rm with}~t\in I_n,
\]
for $1\le n\le N.$  For a smooth  initial data $u_0$, by using  a graded mesh of the form $t_n=(n/N)^\gamma T$ where the exponent $\gamma\ge 1$ chosen appropriately (depends on the regularity of the continuous solution), the numerical results in  \cite{MustaphaAbdallahFuratiNour2016} showed that the above numerical scheme  is second-order accurate in both time and space,   However the theoretical results there  were slightly pessimistic, where an $O(h^2+\tau^{2-\frac{\alpha}{2}})$ error bound  was achieved.


In our test example, $\kappa_\alpha(x,t) =2+\sin(2\pi x)+t^{2+\alpha}$,
$T=1$~and $\Omega = (0,1)$,  and discontinuous initial data given by $u_0(x) =1$ for $x\in[1/4,3/4]$ and $u_0(x)=0$ elsewhere.
Since  $u_0 \in \dot H^\delta(\Omega)$ for  $0\le \delta <1/2,$
by applying Theorem~\ref{thm: smooth and nonsmoothv2} with
$\delta =\tfrac12-\epsilon$~and $\epsilon^{-1}=\log(e^2+t^{-1})$ (so that
$t^{-\epsilon}\le e$~and $0<\epsilon<1/2$), gives
\begin{equation}\label{eq: error ex1}
\|u(t)-u_h(t)\|\le Ct^{-3\alpha/4}h^2\sqrt{\log(e^2+t^{-1})}
	\quad\text{for $0<t\le1$.}
\end{equation}

In our computations,  a uniform spatial mesh with~$h=1/M$ was employed.
In all cases, $M$ was divisible by~$4$ so that the
points~$1/4$~and $3/4$ (where $u_0$ is discontinuous)
coincided with two of the nodes. We first computed a reference
solution~$\Uref^n=U^n$ using a fine spatial mesh with $M=1024$ and a fine time graded mesh of the form $t_n=(n/N)^2$ with $N=5,000.$ We then computed $U^n$ for~$M\in\{16,32,64,128,256\}$, again
with~$N=5,000$.  The initial data was chosen as~$u_{0h}=P_hu_0$ in
each case.   With such a small~$\tau$, for $1\le n\le N$,  the $L^2(\Omega)$-norm error $E_{h,0}^n:=\|U^n-\Uref^n\|$  and the $H^1(\Omega)$-norm error
 $E_{h,1}^n:=\|U^n-\Uref^n\|_1$, were  dominated by the influence of the spatial discretisation. We sought to estimate the $L^2$ convergence rates $\sigma_{h,0}$  and $H^1$ convergence rates $\sigma_{h,1}$ from the relation
$\sigma_{h,\ell}=\log_2(E^*_{2h,\ell}/E^*_{h,\ell}),$
 where the weighted error
\begin{equation}\label{eq: E*}
E^*_{h,\ell}=\max_{1\leq n\leq N}
	\frac{t_n^{3\alpha/4}E_{h,\ell}^n}{\sqrt{\log(e^2+t_n^{-1})}}\quad{\rm for}~~\ell=0,1,
\end{equation}
For three different values of~$\alpha$, Table~\ref{table: errors} and Table \ref{table: errors Ih} show the values of $E^*_{h,0}$~and
$\sigma_{h,0}$, and of $E^*_{h,1}$~and
$\sigma_{h,1}$, respectively. As expected from Theorem~\ref{thm: smooth and nonsmoothv2} and the estimate in \eqref{eq: H1 bound 0 delta 1}, the computed
values of~$\sigma_{h,0}$ and $\sigma_{h,1}$ are close to~$2$ and $1$, respectively. Furthermore,   Figure~\ref{fig: error} shows how the
$L^2$ error $E_{h,0}^n$ (solid lines) and $H^1$ error $E_{h,1}^n$ (dashed lines) vary with~$t_n$ for different mesh size $h$.  Due to
the log-log scale, the graph of a function proportional
to~$t^{-3\alpha/4}$ appears as a straight line with
gradient~$-3\alpha/4$, indicated by the small triangle, and we observe
exactly this behavior of the error for~$t$ (relatively) close to  zero.

\begin{table}
\caption{Weighted $L^2(\Omega)$-norm errors and
convergence rates.}
\label{table: errors}
\begin{center}
\renewcommand{\arraystretch}{1.2}
\begin{tabular}{c|cc|cc|cc}
$M$&
\multicolumn{2}{c|}{$\alpha=0.3$}&
\multicolumn{2}{c|}{$\alpha=0.5$}&
\multicolumn{2}{c}{$\alpha=0.8$}\\
\hline
  16& 1.9861e-04&        & 1.7497e-04&         & 1.5665e-04&         \\
  32& 4.9618e-05&  2.0010& 4.3705e-05&   2.0012& 3.9116e-05&   2.0017\\
  64& 1.2376e-05&  2.0034& 1.0900e-05&   2.0035& 9.7531e-06&   2.0038\\
 128& 3.0654e-06&  2.0134& 2.6995e-06&   2.0136& 2.4135e-06&   2.0147\\
 256& 7.3777e-07&  2.0548& 6.4937e-07&   2.0556& 5.7866e-07&   2.0603
 \end{tabular}
\end{center}
\end{table}

\begin{figure}
\begin{center}
\includegraphics[width=13.0cm,height=7cm]{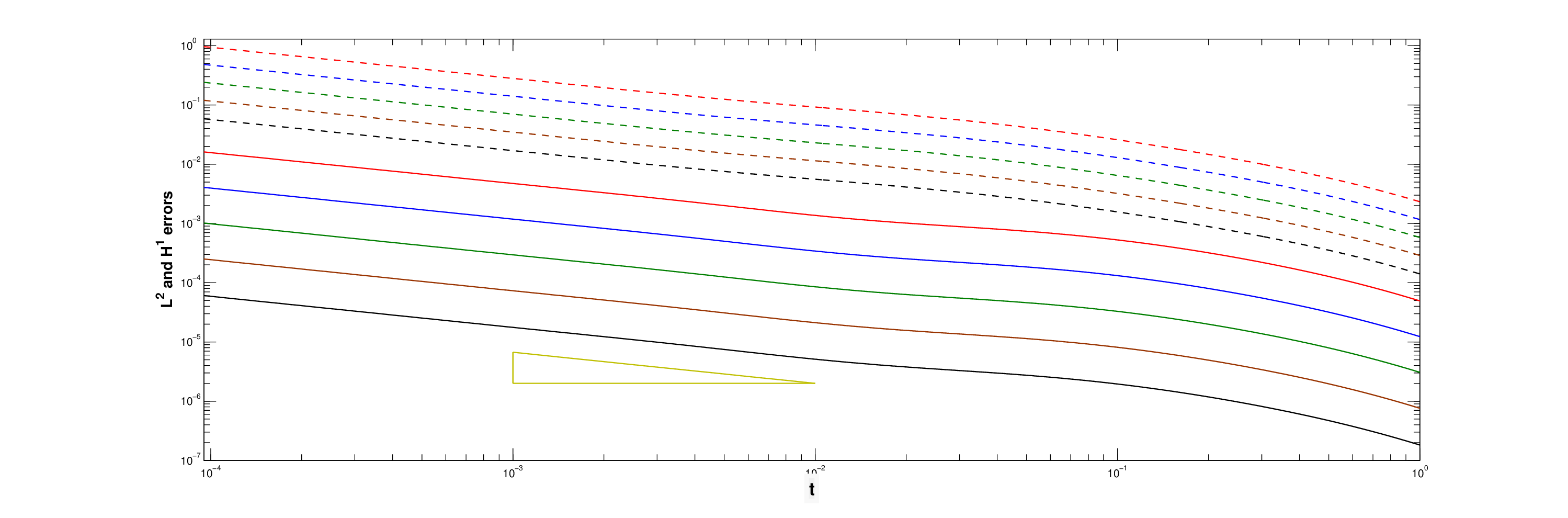}
\caption{Plots of the errors $E^n_{h,0}$ (solid lines) and $E^n_{h,1}$ (dashed lines) as a function of~$t_n$,
for~$\alpha=0.7$ with $M=16,32,64,128,$ and $256$ (in order from top to bottom).
The triangle indicates the gradient~$-3\alpha/4$ for a function
proportional to~$t^{-3\alpha/4}$; cf.~\eqref{eq: error ex1}. } \label{fig: error}
\end{center}
\end{figure}

\begin{table}
\caption{Weighted $H^1(\Omega)$-norm errors and
convergence rates.}
\label{table: errors Ih}
\begin{center}
\renewcommand{\arraystretch}{1.2}
\begin{tabular}{c|cc|cc|cc}
$M$&
\multicolumn{2}{c|}{$\alpha=0.3$}&
\multicolumn{2}{c|}{$\alpha=0.50$}&
\multicolumn{2}{c}{$\alpha=0.8$}\\
\hline
  16& 1.1876e-02&         & 1.0309e-02&         & 8.2688e-03&          \\
  32& 5.9376e-03&   1.0001& 5.1529e-03&   1.0004& 4.1321e-03&   1.0008\\
  64& 2.9647e-03&   1.0020& 2.5727e-03&   1.0021& 2.0629e-03&   1.0022\\
 128& 1.4736e-03&   1.0085& 1.2788e-03&   1.0085& 1.0253e-03&   1.0086\\
 256& 7.1905e-04&   1.0352& 6.2398e-04&   1.0352& 5.0032e-04&   1.0352
  \end{tabular}
\end{center}
\end{table}


\end{document}